\documentclass[12pt,a4paper]{article}
\usepackage[utf8]{inputenc}
\usepackage{amsmath}
\usepackage[numbers]{natbib}
\usepackage{hyperref}
\usepackage{breqn}
\usepackage{amsthm}
\usepackage{amsfonts}
\usepackage{amssymb} 
\usepackage{graphicx}
\usepackage[left=2cm,right=2cm,top=2cm,bottom=2cm]{geometry}
\usepackage{lineno,hyperref}
\modulolinenumbers[5]
\newtheorem{theorem}{Theorem}[section]
\newtheorem{corollary}{Corollary}[section]
\newtheorem{lemma}{Lemma}[section]

\newtheorem*{keyword}{Keywords}
\newtheorem*{remark}{Remark}
\newtheorem{example}{Example}[section]
\newtheorem*{definition}{Definition}

\begin{document}
%\begin{frontmatter}
%Define author and Title
%Note: \thanks{some note} places a note in the footer and an * a the location of this command
\title{Repeated Sums and Binomial Coefficients}
\author{Roudy El Haddad \\
Universit\'e La Sagesse, Facult\'e de g\'enie, Polytech \\
\href{roudy1581999@live.com}{roudy1581999@live.com}}
\date{}
%Place title, author, and date 
\maketitle

%Abstract 
\begin{abstract}
Binomial coefficients have been used for centuries in a variety of fields and have accumulated numerous definitions. In this paper, we introduce a new way of defining binomial coefficients as repeated sums of ones. A multitude of binomial coefficient identities will be shown in order to prove this definition.  Using this new definition, we simplify some particular sums such as the repeated Harmonic sum and the repeated Binomial-Harmonic sum. We derive formulae for simplifying general {\it repeated sums} as well as a variant containing binomial coefficients. 
Additionally, we study the $m$-th difference of a sequence and show how sequences whose $m$-th difference is constant can be related to binomial coefficients.    
\end{abstract}
%Keywords
\begin{keyword}
{\em Binomial Coefficients, Binomial Sums, $m$-th Difference, Repeated Sequences, Repeated Sums, Harmonic Sums, Binomial-Harmonic Sums. } \\  
\thanks{\bf{MSC 2020:}{ 05A10}}
\end{keyword}

%\end{frontmatter}

% 1. Introduction 
\section{Introduction}
The author is interested in the study of the various types of repetitive sums, that is, developing formulae and identities to better understand these types of sums as well as to simplify the way we work with them. Previously, the author has presented a set of useful formulae for two types of repetitive sums: recurrent sums \cite{RecurrentSums} and multiple sums \citep{MultipleSums}. In this article, we aim to do the same for a third type of repetitive sums: the repeated sums. For the first two types, partitions were key for developing a reduction theorem to simplify these sums. For this type of sums, binomial coefficients are key. Hence, we first need to develop a set of binomial coefficient identities as well as we need to introduce a new definition for binomial coefficients. 
Binomial coefficients appear in many distinct fields of mathematics (including probability, combinatorics, analysis, etc.) and, therefore, they can be defined in many different ways. 
In 1654, in his article ``Trait{\'e} du triangle arithm{\'e}tique'' \cite{pascal1978traite,pascal2011traite}, Pascal presented what is now known as Pascal's triangle and which constituted the first definition of binomial coefficients.
The binomial coefficient $\binom{n}{k}$ represents the element in the $k$-th column and $n$-th row of Pascal’s triangle. They can also be defined by the recurrent relation used to construct Pascal's triangle which we will call the Pascal's triangle identity and which states that $\binom{n}{k}=\binom{n-1}{k}+\binom{n-1}{k-1}$. 
The notation $\binom{n}{k}$ was introduced by Euler in the 18th century. 
Alternatively, binomial coefficients can be denoted as $C_n^k$. This second notation, introduced in the 19th century, is mostly used in combinatorics where binomial coefficients are defined as the number of ways of choosing an unordered subset of $k$ elements from a set of $n$ elements. 
Binomial coefficients are also intensively used in probability calculation for binomial distributions. Let X be a random variable, let $n$ be the number of experiments, let $k$ be an integer, and let $p$ represent the probability of success in a single experiment, the binomial distribution formula is given by 
$
P(X=k)=C_n^k p^{k}(1-p)^{n-k}
$. 
Additional definitions of these coefficients appear in analysis and calculus. First, they appear in the binomial theorem as well as in the generalized binomial theorem presented by Newton \cite{whiteside1961newton,whiteside1961henry}, 
$
(u+v)^n=\sum_{k=0}^{n}{\binom{n}{k}u^{k}v^{n-k}}
$
and 
$
(u_1+ \cdots +u_m)^n=\sum_{k_1+\cdots+k_m=n}{\binom{n}{k_m}\binom{n-k_m}{k_{m-1}}\cdots\binom{n-k_m-\cdots-k_1}{k_2}u_1^{k_1}u_2^{k_2}\cdots u_m^{k_m}}
$.  
From the simple case of the theorem, the binomial coefficient $\binom{n}{k}$ can be defined as the coefficient of $u^k v^{n-k}$ in the expansion of $(u+v)^n$. 
Likewise, $\binom{n}{k}$ appears in Leibniz's formula for the $n$-th derivative of a product \cite{polking1972leibniz,mazkewitsch1963n,dybowski2009generalization}, 
$
(uv)^{(n)}=\sum_{k=0}^{n}{\binom{n}{k}u^{(k)}v^{(n-k)}}
$. 
Hence, Leibniz's rule defines the binomial coefficient $\binom{n}{k}$ as the coefficient of $u^{(k)}v^{(n-k)}$ in the expansion of $(uv)^{(n)}$. 
These coefficients even appear in number theory: In \cite{RecurrentSums,MultipleSums}, binomial coefficients are defined as a sum over partitions of an integer $m$ in several ways. 
And, finally, binomial coefficients can be defined in terms of factorials: $\binom{n}{k}=\frac{n!}{k!(n-k)!}$. 
\begin{remark}
{\em Tables of binomial coefficients can be found in \cite{bol1983mathematical,miller1954table,royal1954mathematical}.  }
\end{remark}
In this study, we propose a new definition: we define binomial coefficients as repeated sums (Section \ref{Binomial Coefficients as Repeated Sums}). This definition is key for the determination of a reduction formula for repeated sums which allows the reduction of such sums into simple non-repetitive sums. 
In Section \ref{Binomial Coefficients as Repeated Sums}, we also present formulas for the $m$-th difference of a sequence as well as we show the link between sequences with a constant $m$-th difference and binomial coefficients. 
In Section \ref{Sums of Binomial Coefficients}, we present some formulae for binomial sums and repeated binomial sums which are needed to prove the definition. 
In Section \ref{Reduction of Binomial-Harmonic Sums}, this new definition is used to derive results related to the harmonic sum such as an expression for the repeated harmonic sum in terms of the simple harmonic sum. We also derive expressions for a modified version of the repeated harmonic sum which we will refer to as the repeated binomial-harmonic sum. Then, in Section \ref{Application to the Calculation of Repeated Sums}, this definition as well as the results derived will be utilized to develop a reduction formula for the repeated sum of any sequence $a_N$. A similar expression is also derived for a modified version of repeated sums which we will refer to as the repeated “binomial-sequence” sum. 
% 2. Variation Formulas 
\section{Sums of binomial coefficients} \label{Sums of Binomial Coefficients}
In this section, we prove some formulas related to the sums and repeated sums of binomial coefficients. These formulas are needed in order to produce the proposed definition for binomial coefficients. They are also essential for developing the harmonic sum identities later on. Furthermore, they are crucial for developing an expression for repeated sums in terms of single sums. 
\subsection{Binomial sum}
We begin by determining a formula for the sum of binomial coefficients. 
%Theorem 2.1 
\begin{theorem} \label{t 2.1}
For any $k,q,n \in \mathbb{N}$ such that $n \geq q$, we have 
$$
\sum_{i=q}^{n}{\binom{i}{k}}
=\binom{n+1}{k+1}-\binom{q}{k+1}
.$$
\end{theorem}
\begin{remark}
{\em The coefficients of the form $\binom{n}{k}$ where $n<k$, being undefined according to the factorial definition, are supposed zero. Hence, if $q<k$, we can start the sum at $k$. }
\end{remark}
\begin{proof}
1. Base case: verify true for $n=q$. \\
Using Pascal's Triangle identity,  
$$
\sum_{i=q}^{q}{\binom{i}{k}}
=\binom{q}{k}
=\binom{q+1}{k+1}-\binom{q}{k+1}
.$$
2. Induction hypothesis: assume the statement is true until $n$. 
$$
\sum_{i=q}^{n}{\binom{i}{k}}
=\binom{n+1}{k+1}-\binom{q}{k+1}
.$$
3. Induction step: we will show that this statement is true for $(n+1)$. \\
We have to show the following statement to be true:  
$$
\sum_{i=q}^{n+1}{\binom{i}{k}}
=\binom{n+2}{k+1}-\binom{q}{k+1}
.$$
$ \\ $ 
$$
\sum_{i=q}^{n+1}{\binom{i}{k}}
=\binom{n+1}{k}+\sum_{i=q}^{n}{\binom{i}{k}}
=\binom{n+1}{k}+\binom{n+1}{k+1}-\binom{q}{k+1}
.$$ 
From Pascal's Triangle identity, we get the case for $(n+1)$. 
\end{proof}
%Corollary 2.1 
\begin{corollary} \label{c 2.1}
For $q=0$ or $q=k$, Theorem \ref{t 2.1} becomes 
$$
\sum_{i=0}^{n}{\binom{i}{k}}
=\sum_{i=k}^{n}{\binom{i}{k}}
=\binom{n+1}{k+1}
.$$
\end{corollary}
%Corollary 2.2 
\begin{corollary} \label{c 2.2}
The shifted binomial sum can be calculated as follows 
$$
\sum_{i=m}^{n+m}{\binom{i}{k}}
=\sum_{i=0}^{n}{\binom{i+m}{k}}
=\binom{n+m+1}{k+1}-\binom{m}{k+1}
.$$
\end{corollary}
\begin{proof}
Applying Theorem \ref{t 2.1} for $q=m$ and replacing $n$ by $n+m$, we obtain the corollary. 
\end{proof}
\subsection{Repeated binomial sum}
Using Theorem \ref{t 2.1}, we develop a formula for the repeated sum of binomial coefficients. 
%Theorem 2.2 
\begin{theorem} \label{t 2.2}
For any $m \in \mathbb{N^*}$ and for any $k,q,n \in \mathbb{N}$ such that $n \geq q$, we have 
\begin{equation*}
\begin{split}
\sum_{N_m=q}^{n}{\cdots \sum_{N_1=q}^{N_2}{\binom{N_1}{k}}}
&=\binom{n+m}{k+m}
-\sum_{j=1}^{m}{\binom{(q-1)+j}{k+j}\binom{(n-q)+(m-j)}{m-j}} \\
&=\binom{n+m}{k+m}
-\sum_{j=1}^{m}{\binom{(q-1)+j}{(q-1)-k}\binom{(n-q)+(m-j)}{n-q}} 
.\end{split}
\end{equation*}
\end{theorem}
\begin{proof}
1. Base case: verify true for $m=1$. 
$$
\binom{n+1}{k+1}
-\sum_{j=1}^{1}{\binom{(q-1)+j}{k+j}\binom{(n-q)+(1-j)}{1-j}}
=\binom{n+1}{k+1}
-\binom{q}{k+1}
.$$
From Theorem \ref{t 2.1}, this case is proven. \\
2. Induction hypothesis: assume the statement is true until $m$. 
$$
\sum_{N_m=q}^{n}{\cdots \sum_{N_1=q}^{N_2}{\binom{N_1}{k}}}
=\binom{n+m}{k+m}
-\sum_{j=1}^{m}{\binom{(q-1)+j}{k+j}\binom{(n-q)+(m-j)}{m-j}}
.$$
3. Induction step: we will show that this statement is true for $(m+1)$. \\
We have to show the following statement to be true:   
$$
\sum_{N_{m+1}=q}^{n}{\cdots \sum_{N_1=q}^{N_2}{\binom{N_1}{k}}}
=\binom{n+m+1}{k+m+1}
-\sum_{j=1}^{m+1}{\binom{(q-1)+j}{k+j}\binom{(n-q)+(m+1-j)}{m+1-j}}
.$$
$$ \\ $$ 
\begin{equation*}
\begin{split}
\sum_{N_{m+1}=q}^{n}{\cdots \sum_{N_1=q}^{N_2}{\binom{N_1}{k}}}
&=\sum_{N_{m+1}=q}^{n}{\left(\sum_{N_{m}=q}^{N_{m+1}}{\cdots \sum_{N_1=q}^{N_2}{\binom{N_1}{k}}}\right)} \\
&\resizebox{\linewidth-108pt}{!}{$\displaystyle{=\sum_{N_{m+1}=q}^{n}{\binom{N_{m+1}+m}{k+m}}
-\sum_{N_{m+1}=q}^{n}{\sum_{j=1}^{m}{\binom{(q-1)+j}{k+j}\binom{(N_{m+1}-q)+(m-j)}{m-j}}} }$} \\
&\resizebox{\linewidth-108pt}{!}{$\displaystyle{=\sum_{N_{m+1}=q}^{n}{\binom{N_{m+1}+m}{k+m}}
-\sum_{j=1}^{m}{\sum_{N_{m+1}=q}^{n}{\binom{(q-1)+j}{k+j}\binom{(N_{m+1}-q)+(m-j)}{m-j}}} }$}\\
%&=\sum_{N_{m+1}=q+m}^{n+m}{\binom{N_{m+1}}{k+m}}-\sum_{j=1}^{m}{\sum_{N_{m+1}=m-j}^{(n-q)+(m-j)}{\binom{(q-1)+j}{k+j}\binom{N_{m+1}}{m-j}}}\\
&=\sum_{N_{m+1}=q+m}^{n+m}{\binom{N_{m+1}}{k+m}}
-\sum_{j=1}^{m}{\binom{(q-1)+j}{k+j}\sum_{N_{m+1}=m-j}^{(n-q)+(m-j)}{\binom{N_{m+1}}{m-j}}}.
\end{split}
\end{equation*}
By applying Theorem \ref{t 2.1}, 
\begin{equation*}
\begin{split}
\sum_{N_{m+1}=q}^{n}{\cdots \sum_{N_1=q}^{N_2}{\binom{N_1}{k}}}
=&\binom{n+m+1}{k+m+1}-\binom{q+m}{k+m+1} \\ 
&-\sum_{j=1}^{m}{\binom{(q-1)+j}{k+j}\binom{(n-q)+(m+1-j)}{m+1-j}} 
.\end{split}
\end{equation*}
Noticing that 
\begin{equation*}
\sum_{j=m+1}^{m+1}{\binom{(q-1)+j}{k+j}\binom{(n-q)+(m+1-j)}{m+1-j}}
=\binom{q+m}{k+m+1},
\end{equation*}
hence, the case for $(n+1)$ is proven and the theorem is proven by induction.  
\end{proof}
%Corollary 2.3 
An important particular case of Theorem \ref{t 2.2} is illustrated by the following corollary. In fact, the formula for the repeated sum of binomial coefficients is heavily simplified if the sums are started at $0$. 
\begin{corollary} \label{c 2.3}
For any $m \in \mathbb{N^*}$ and for any $k,n \in \mathbb{N}$, we have that 
$$
\sum_{N_{m}=0}^{n}{\cdots \sum_{N_1=0}^{N_2}{\binom{N_1}{k}}}
=\binom{n+m}{k+m}
.$$
\end{corollary}
\begin{proof}
From Theorem \ref{t 2.2} for $q=0$, 
$$
\sum_{N_{m}=0}^{n}{\cdots \sum_{N_1=0}^{N_2}{\binom{N_1}{k}}}
=\binom{n+m}{k+m}
-\sum_{j=1}^{m}{\binom{j-1}{k+j}\binom{n+(m-j)}{m-j}}
.$$
We can notice that because $k \geq 0$, then $k+j \geq j>j-1$, thus, $\forall j \in \mathbb{N}$, $\binom{j-1}{k+j}=0$. Hence, we obtain the corollary. 
\end{proof}
The formula for repeated sums of ones, which is needed to prove the new definition for binomial coefficients, is another crucial particular case of Theorem \ref{t 2.2}. 
%Corollary 2.4 
\begin{corollary} \label{c 2.4}
For any $m \in \mathbb{N^*}$ and for any $q,n \in \mathbb{N}$ such that $n \geq q$, we have that 
$$
\sum_{N_{m}=q}^{n}{\cdots \sum_{N_1=q}^{N_2}{1}}
=\binom{n-q+m}{m}
.$$
\end{corollary}
\begin{proof}
$$
\sum_{N_{m}=q}^{n}{\cdots \sum_{N_1=q}^{N_2}{1}}
=\sum_{N_{m}=q}^{n}{\cdots \sum_{N_2=q}^{N_3}{\sum_{N_1=0}^{N_2-q}{1}}}
=\sum_{N_{m}=q}^{n}{\cdots \sum_{N_2=0}^{N_3-q}{\sum_{N_1=0}^{N_2}{1}}}
=\cdots
=\sum_{N_{m}=0}^{n-q}{\cdots \sum_{N_2=0}^{N_3}{\sum_{N_1=0}^{N_2}{1}}}
.$$
From Corollary \ref{c 2.3} for $k=0$, we obtain the corollary.  
\end{proof}
The sum of a special product of binomial coefficients can also be derived from Theorem \ref{t 2.2}. 
%Corollary 2.5 
\begin{corollary} \label{c 2.5}
For any $m \in \mathbb{N^*}$ and for any $q,n \in \mathbb{N}$ such that $n \geq q$, we have that 
\begin{equation*}
\begin{split}
\sum_{j=1}^{m}{\binom{(q-1)+j}{j}\binom{(n-q)+(m-j)}{m-j}}
&=\sum_{j=1}^{m}{\binom{(q-1)+j}{q-1}\binom{(n-q)+(m-j)}{n-q}} \\
&=\binom{n+m}{m}-\binom{n-q+m}{m}
.\end{split}
\end{equation*}
\end{corollary}
\begin{proof}
By applying Theorem \ref{t 2.2} for $k=0$, 
\begin{equation*}
\begin{split}
\sum_{N_m=q}^{n}{\cdots \sum_{N_1=q}^{N_2}{1}}
&=\binom{n+m}{m}
-\sum_{j=1}^{m}{\binom{(q-1)+j}{j}\binom{(n-q)+(m-j)}{m-j}} \\
&=\binom{n+m}{m}
-\sum_{j=1}^{m}{\binom{(q-1)+j}{q-1}\binom{(n-q)+(m-j)}{n-q}} 
.\end{split}
\end{equation*}
By comparing with the expression given by Corollary \ref{c 2.4}, we get this corollary. 
\end{proof}
\section{Binomial coefficients as repeated sums and repeated sequences} \label{Binomial Coefficients as Repeated Sums}
In this section, we use the formulas presented in the previous section in order to introduce a new definition of binomial coefficients. First, we prove that a binomial coefficient can be defined as a repeated sum of `1's. Second, we prove that a binomial coefficient can be defined as a repeated sequence. 
\subsection{Binomial coefficients as repeated sums of ones}
In this section, we prove that binomial coefficients can be defined as repeated sums of ones. \\ 
A repeated sum of `1's can be expressed as a binomial coefficient as is shown by Theorem \ref{t 3.1}. 
%Theorem 3.1
\begin{theorem} \label{t 3.1}
A repeated sum of 1s corresponds to the following binomial coefficient,  
$$
\sum_{N_m=1}^{n}{\cdots \sum_{N_1=1}^{N_2}{1}}
=\binom{n+m-1}{m}
.$$
$$
\sum_{N_m=0}^{n}{\cdots \sum_{N_1=0}^{N_2}{1}}
=\binom{n+m}{m}
.$$
\end{theorem}
\begin{proof}
The first equality of this theorem is obtained by applying Corollary \ref{c 2.4} for $q=1$ and the second equality is obtained by applying Corollary \ref{c 2.4} for $q=0$. 
\end{proof}
\begin{remark}
{\em We can highlight the combinatorical nature of repeated sums of 1s as $\binom{n-1+m}{m}$ represents the number of ways of choosing $m$ elements from $n$ with replacement but without regard to order  (\cite{olofsson2012probability} pp. 24-25).}
\end{remark}
%Theorem 3.2
%\begin{theorem} \label{t 3.2}
%A repeated sum of 1s corresponds to the following binomial coefficient,  
%$$
%\sum_{N_m=0}^{n}{\cdots \sum_{N_1=0}^{N_2}{1}}
%=\binom{n+m}{m}
%.$$
%\end{theorem}
%\begin{proof}
%This theorem is obtained by applying Corollary \ref{c 2.4} for $q=0$. 
%\end{proof}
A binomial coefficient can be defined as a repeated sum of 1s as shown by Theorem \ref{t 3.3}. 
%Theorem 3.3 
\begin{theorem} \label{t 3.3}
A binomial coefficient corresponds to the following repeated sums of 1s, 
$$
\binom{n}{k}
=\sum_{N_k=1}^{n-k+1}{\cdots \sum_{N_1=1}^{N_2}{1}}
=\sum_{N_k=0}^{n-k}{\cdots \sum_{N_1=0}^{N_2}{1}}
.$$
\end{theorem}
\begin{proof}
The first part of this theorem is obtained by replacing $n$ by $n-k+1$ in the first equation of Theorem \ref{t 3.1}. The second part is obtained by replacing $n$ by $n-k$ in the second equation of Theorem \ref{t 3.1}.
\end{proof}
\subsection{Binomial coefficients as repeated sequences} 
In this section, we will go through a series of lemmas in order to prove that any binomial coefficient can be defined as a repeated sequence (defined below). In order to do so, we will first prove that repeated sequences can be explicitly expressed as a repeated sum of ones. Then, using the fact that any repeated sum of ones can be defined as a binomial coefficient (as proven in the previous section), we define these repeated sequences as binomial coefficients. 
\subsubsection{Notation and definitions}
%Definition 3.1
\begin{definition}
We define the first difference of a sequence $(x_n)$, denoted by $\Delta x_n$, as the difference between two consecutive terms of the sequence $(x_n)$:  
$$
\Delta x_{n} = x_{n} - x_{n-1}
.$$
\end{definition}
%Definition 3.2
\begin{definition}
Let $m \in \mathbb{N}$. We define the $m$-th difference of a sequence $(x_n)$, denoted by $\Delta^m x_n$, by the following recurrent relation:  
$$
\begin{cases}
\Delta^{m} x_{n} = \Delta^{m-1} x_{n} - \Delta^{m-1} x_{n-1} \\
\forall i \in \mathbb{Z}, \Delta^{0} x_{i} = x_{i}.
\end{cases}
$$
%\begin{remark}
%{\em Explicitly, the definition would become $\Delta^{m} x_{n} 
%= \sum_{j=0}^{m}{\binom{m}{j}x_{n-j}}
%= \sum_{j=0}^{p}{\binom{p}{j} \Delta ^{m-p} x_{n-j}}$}.
%\end{remark}
\end{definition}
%Definition 3.3
%\begin{definition}
%We define a general repeated sequence (or repeated sequence) of degree $m$, denoted by $\{x_{n,m}\}_{n=1}^{\infty}$, as a sequence whose $m$-th difference is constant $(\Delta^m x_{n,m}$ is a constant$)$. $x_{n,m}$ is the $n$-th term of the repeated sequence of order $m$. 
%\end{definition}
%Two particular cases of interest are the following: 
%Definition 3.4
\begin{definition}
We define a repeated sequence of degree $m$, denoted by $\{x_{n,m}\}_{n=1}^{\infty}$, as a sequence whose $m$-th difference is a constant $(\forall i,j \in \mathbb{N}^*, \,\, \Delta^m x_{i,m}=\Delta^m x_{j,m})$.
% and whose first non-zero term is $\overline{x}_{1,m}=\Delta^m \overline{x}_{n,m}$ $(\forall j < 1, \overline{x}_{j,m}=0)$. 
 $x_{n,m}$ is the $n$-th term of this repeated sequence of order $m$. 
% In other words, it is a particular case of the general repeated sequence where the extension is defined to be zeros. 
\end{definition}
\begin{remark}
{\em In order to keep the $m$-th difference always valid we define an extension of the sequence such that $\forall i \leq 0, x_{i,m}=0$. However, the properties of the sequence do not extend to this extension. Note that the theorems develop in this section will only apply to repeated sequences which  validate the condition imposed by this extension.}
\end{remark}
\begin{remark}
{\em For a given order $m$ and a given $m$-th difference, the repeated sequence is unique. The repeated sequence of order $m$ having $\Delta^m x_{n,m}=\delta \in \mathbb{C}$ is denoted by $\{x_{n,m}^{(\delta)}\}_{n=1}^{\infty}$. 
%$x_{n,m}^{(\delta)}$ $x_{n,m,\delta}$
%A given repeated sequence of degree $m$ is defined by its $m$-th difference. The repreated sequence of degree $m$ having $\Delta^m x_{n,m}$
}
\end{remark}
%Definition 3.5
\begin{definition}
We define the unity repeated sequence, denoted by $\{\hat{x}_{n,m}\}_{n=1}^{\infty}$, as the sequence whose $m$-th difference is 1 $(\forall i \in \mathbb{N}^*, \,\, \Delta^m \hat{x}_{i,m}=1)$.  
%and whose initial term is 1 $(\hat{x}_{1,m}=1$ and $\forall j \leq 0, \hat{x}_{j,m}=0)$. 
$\hat{x}_{n,m}$ is the $n$-th term of the unity repeated sequence of order $m$. 
\end{definition}
%\textbf{Examples} 
\begin{example}
{\em The unity repeated sequence of order $0$, $(\hat{x}_{n,0})$, is the following: }
$$
\hat{x}_{1,0}=1, \,\, \hat{x}_{2,0}=1, \,\, \hat{x}_{3,0}=1, \cdots
$$
\end{example}
\begin{example}
{\em The unity repeated sequence of order $1$, $(\hat{x}_{n,1})$, is the following: }
$$
\hat{x}_{1,1}=1, \,\, \hat{x}_{2,1}=2(=1+1), \,\, \hat{x}_{3,1}=3(=1+1+1), \cdots
$$
\end{example}
\begin{example}
{\em The unity repeated sequence of order $2$, $(\hat{x}_{n,2})$, is the following: }
$$
\hat{x}_{1,2}=1, \,\, \hat{x}_{2,2}=3(=1+(1+1)), \,\, \hat{x}_{3,2}=6(=1+(1+1)+(1+1+1)), \cdots
$$
\end{example}
\subsubsection{Expressions for the m-th difference of a sequence}
In this section, using the recurrent definition of the $m$-th difference, we develop an additional definition for the $m$-th difference of a sequence in terms of a lower order difference of this sequence. 
\begin{theorem} \label{explicitVariation1}
For any $m,n,p \in \mathbb{N}$ such that $p \leq m$ and for any sequence $x_n$, we have that 
$$\Delta^{m} x_{n} 
= \sum_{j=0}^{p}{(-1)^{j}\binom{p}{j} \Delta ^{m-p} x_{n-j}}.$$
\end{theorem}
\begin{proof} 
1. Base case: verify true for $m=0$. \\ 
Knowing that $0 \leq p \leq m$ and $m=0$, hence, $p=0$. 
$$
\sum_{j=0}^{0}{(-1)^{j}\binom{0}{j} \Delta^{0} x_{n-j}}=\Delta^{0} x_{n}.
$$
2. Induction hypothesis: assume the statement is true until $m$. 
$$\Delta^{m} x_{n} 
= \sum_{j=0}^{p}{(-1)^{j}\binom{p}{j} \Delta ^{m-p} x_{n-j}}.$$ 
3. Induction step: we will show that this statement is true for $(m+1)$. \\
We have to show the following statement to be true:   
$$\Delta^{m+1} x_{n} 
= \sum_{j=0}^{p}{(-1)^{j}\binom{p}{j} \Delta^{m-p+1} x_{n-j}}.$$
$$ \\ $$
Knowing that $\Delta^{m+1}x_n=\Delta^{m}x_n-\Delta^{m}x_{n-1}$ and applying the induction hypothesis, we get 
\begin{equation*}
\begin{split}
\Delta^{m+1}x_n
%&=\Delta^{m}x_n-\Delta^{m}x_{n-1} \\
&=\sum_{j=0}^{p}{(-1)^{j}\binom{p}{j} \Delta ^{m-p} x_{n-j}}
-\sum_{j=0}^{p}{(-1)^{j}\binom{p}{j} \Delta ^{m-p} x_{n-j-1}}\\
&=\sum_{j=0}^{p}{(-1)^{j}\binom{p}{j} (\Delta ^{m-p} x_{n-j}-\Delta ^{m-p} x_{n-j-1})} \\
&=\sum_{j=0}^{p}{(-1)^{j}\binom{p}{j} \Delta ^{m-p+1} x_{n-j}} 
.\end{split}
\end{equation*}
\end{proof}
\begin{remark}
{\em As we can see, the binomial coefficient $\binom{p}{j}$ can be defined as the coefficient of $(-1)^j\Delta^{m-p}x_{n-j}$ in the expression of the $m$-th difference of the sequence $(x_n)$.}
\end{remark}
In particular, the $m$-th difference of a sequence can be expressed in terms of its elements. 
\begin{corollary} \label{explicitVariation2}
For $p=m$, Theorem \ref{explicitVariation1} becomes 
$$\Delta^{m} x_{n} 
= \sum_{j=0}^{m}{(-1)^{j}\binom{m}{j}x_{n-j}}.
$$
\end{corollary} 
\subsubsection{Explicit expression for a repeated sequence as a repeated sum of ones}
We are interested in the calculation of the sum of such sequences as well as in the definition of such sequences in terms of binomial coefficients. To do so, in this section, we convert the definition of a repeated sequence of degree $m$ which is in terms of the $m$-th difference into a definition in terms of a repeated sum of ones. Then, in the later sections, the theorems for repeated sums of ones are used to prove the proposed definition as well as to develop a formula for the repeated sum of a repeated sequence of degree $m$. \\

We start by proving the following set of lemmas.  
%Lemma 3.1
\begin{lemma} \label{l 3.1}
For repeated sequences of any order $k$, the difference of any order $m$ is zero for all terms with an index $i \leq 0$, 
$$
\forall k \in \mathbb{N}, \forall m \in \mathbb{N},\forall i \leq 0, \Delta^m x_{i,k}^{(\delta)}=0 
.$$
\end{lemma}
\begin{proof}
1. Base case: verify true for $m=0$. \\
By definition, $x_{i,k}^{(\delta)}=0$ for $i \leq 0$. Hence, 
$$
\forall k \in \mathbb{N},\forall i \leq 0, 
\Delta^0 x_{i,k}^{(\delta)} = x_{i,k}^{(\delta)} = 0 
.$$
2. Induction hypothesis: assume the statement is true until $m$. 
$$
\forall k \in \mathbb{N},\forall i \leq 0, \Delta^m x_{i,k}^{(\delta)}=0 
.$$
3. Induction step: we will show that this statement is true for $(m+1)$. \\ 
We have to show the following statement to be true:  
$$
\forall k \in \mathbb{N},\forall i \leq 0, \Delta^{m+1} x_{i,k}^{(\delta)}=0 
.$$ 
By definition, 
$
\Delta^{m+1} x_{i,k}^{(\delta)} = \Delta^{m} x_{i,k}^{(\delta)} - \Delta^{m} x_{i-1,k}^{(\delta)} 
$. \\
From the induction hypothesis, we get that $\Delta^{m} x_{i,k}^{(\delta)}=0$ and $\Delta^{m} x_{i-1,k}^{(\delta)}=0$ because, respectively, $i \leq 0$ and $i-1 \leq -1 \leq 0$. \\
Hence, $\Delta^{m+1} x_{i,k}^{(\delta)}=0$.    
\end{proof}
%Lemma 3.2
\begin{lemma} \label{l 3.2}
For repeated sequences of any order $k$, the difference of any order $m$ of the first term of the sequence is equal to this initial term, 
%$$
%\forall k \in \mathbb{N},\forall m \in \mathbb{N}, \Delta^m x_{1,k}=x_{1,k}
%.$$
%$$
%\forall k \in \mathbb{N},\forall m \in \mathbb{N}, \Delta^m x_{1,k,\delta}=x_{1,k,\delta}
%.$$
$$
\forall k \in \mathbb{N},\forall m \in \mathbb{N}, \Delta^m x_{1,k}^{(\delta)}=x_{1,k}^{(\delta)}
.$$
\end{lemma}
\begin{proof}
1. Base case: verify true for $m=0$.  
$$
\forall k \in \mathbb{N}, \Delta^0 x_{1,k}^{(\delta)} =x_{1,k}^{(\delta)}.$$
2. Induction hypothesis: assume the statement is true until $m$. 
$$
\forall k \in \mathbb{N}, \Delta^m x_{1,k}^{(\delta)}=x_{1,k}^{(\delta)}
.$$
3. Induction step: we will show that this statement is true for $(m+1)$. \\ 
We have to show the following statement to be true:  
$$
\forall k \in \mathbb{N}, \Delta^{m+1} x_{1,k}^{(\delta)}=x_{1,k}^{(\delta)}
.$$ 
By definition, 
$
\Delta^{m+1} x_{1,k}^{(\delta)} = \Delta^{m} x_{1,k}^{(\delta)} - \Delta^{m} x_{0,k}^{(\delta)} 
$. \\
From Lemma \ref{l 3.1}, $\Delta^{m} x_{0,k}^{(\delta)} =0$.
Hence, 
$
\Delta^{m+1} x_{1,k}^{(\delta)} = \Delta^{m} x_{1,k}^{(\delta)}
$. \\
By applying the induction hypothesis, we get 
$\Delta^{m+1} x_{1,k}^{(\delta)} =x_{1,k}^{(\delta)}$.
\end{proof} 
\begin{remark}
{\em From Lemma \ref{l 3.2} for $k=m$, we see that the first term of a repeated sequence of order $m$ is equal to its $m$-th difference ($x_{1,m}^{(\delta)}=\Delta^m x_{1,m}^{(\delta)}=\delta$).}
\end{remark}
%Lemma 3.3
\begin{lemma} \label{l 3.3}
For any $m,k,n \in \mathbb{N}$ where $n \geq 1$, we have that 
$$
\sum_{N=1}^{n}{\Delta^{m} x_{N,k}^{(\delta)}}
=\Delta^{m-1} x_{n,k}^{(\delta)}
.$$
\end{lemma}
\begin{proof}
1. Base case: verify true for $n=1$. \\
From Lemma \ref{l 3.2}, $\Delta^{m} x_{1,k}^{(\delta)}=\Delta^{m-1} x_{1,k}^{(\delta)}=x_{1,k}^{(\delta)}$, 
$$
\sum_{N=1}^{1}{\Delta^{m} x_{N,k}^{(\delta)}}
=\Delta^{m} x_{1,k}^{(\delta)}
=\Delta^{m-1} x_{1,k}^{(\delta)}
.$$
2. Induction hypothesis: assume the statement is true until $n$. 
$$
\sum_{N=1}^{n}{\Delta^{m} x_{N,k}^{(\delta)}}
=\Delta^{m-1} x_{n,k}^{(\delta)}
.$$
3. Induction step: we will show that this statement is true for $(n+1)$. \\ 
We have to show the following statement to be true:  
$$
\sum_{N=1}^{n+1}{\Delta^{m} x_{N,k}^{(\delta)}}
=\Delta^{m-1} x_{n+1,k}^{(\delta)}
.$$
$$ \\ $$ 
$$
\sum_{N=1}^{n+1}{\Delta^{m} x_{N,k}^{(\delta)}}
=\sum_{N=1}^{n}{\Delta^{m} x_{N,k}^{(\delta)}}+\Delta^{m} x_{n+1,k}^{(\delta)}
.$$
Applying the induction hypothesis to the first term and using the definition of the second term, 
$$
\sum_{N=1}^{n+1}{\Delta^{m} x_{N,k}^{(\delta)}}
=\Delta^{m-1} x_{n,k}^{(\delta)}+\left[\Delta^{m-1} x_{n+1,k}^{(\delta)}-\Delta^{m-1} x_{n,k}^{(\delta)}\right]
=\Delta^{m-1} x_{n+1,k}^{(\delta)}
.$$
\end{proof}
%Lemma 3.4
\begin{lemma} \label{l 3.4}
For any $m,k,j,n \in \mathbb{N}$ such that $n \geq 1$ and $1 \leq j \leq m$, we have that 
%$$
%\sum_{N_j=1}^{n}{\cdots \sum_{N_2=1}^{N_3}{\sum_{N_1=1}^{N_2}{\Delta^{m} x_{N_1,k}}}}
%=\Delta^{m-j} x_{n,k}
%.$$
$$
\sum_{N_j=1}^{n}{\cdots \sum_{N_2=1}^{N_3}{\sum_{N_1=1}^{N_2}{\Delta^{m} x_{N_1,k}^{(\delta)}}}}
=\Delta^{m-j} x_{n,k}^{(\delta)}
.$$
%$$
%\sum_{N_j=1}^{n}{\cdots \sum_{N_2=1}^{N_3}{\sum_{N_1=1}^{N_2}{\Delta^{m} x_{N_1,k,\delta}}}}
%=\Delta^{m-j} x_{n,k,\delta}
%.$$
\end{lemma}
\begin{proof}
1. Base case: verify true for $j=1$. \\
Proven true in Lemma \ref{l 3.3}. \\
2. Induction hypothesis: assume the statement is true until $j$. 
$$
\sum_{N_j=1}^{n}{\cdots \sum_{N_1=1}^{N_2}{\Delta^{m} x_{N_1,k}^{(\delta)}}}
=\Delta^{m-j} x_{n,k}^{(\delta)}
.$$
3. Induction step: we will show that this statement is true for $(j+1)$. \\ 
We have to show the following statement to be true:  
$$
\sum_{N_{j+1}=1}^{n}{\cdots \sum_{N_1=1}^{N_2}{\Delta^{m} x_{N_1,k}^{(\delta)}}}
=\Delta^{m-j-1} x_{n,k}^{(\delta)}
.$$
Using the induction hypothesis, 
$$
\sum_{N_{j+1}=1}^{n}{\cdots \sum_{N_1=1}^{N_2}{\Delta^{m} x_{N_1,k}^{(\delta)}}}
=\sum_{N_{j+1}=1}^{n}{\sum_{N_{j}=1}^{N_{j+1}}{\cdots \sum_{N_1=1}^{N_2}{\Delta^{m} x_{N_1,k}^{(\delta)}}}}
=\sum_{N_{j+1}=1}^{n}{\Delta^{m-j} x_{N_{j+1},k}^{(\delta)}}
.$$
Applying Lemma \ref{l 3.3}, we get the lemma.  
\end{proof}
A particular case of Lemma \ref{l 3.4} for $j=k=m$ is as follows. This theorem relates the repeated sequence to the repeated sum of ones.  
%%Corollary 3.1
%\begin{corollary}
%$$
%\sum_{N_m=1}^{n}{\cdots \sum_{N_2=1}^{N_3}{\sum_{N_1=1}^{N_2}{\Delta^{m} x_{N_1,m}}}}
%= x_{n,m}
%$$
%\end{corollary}
%Theorem 3.4
\begin{theorem} \label{t 3.4}
Any element of the $m$-th degree repeated sequence can be defined as an $m$-th order repeated sum of ones, 
$$
x_{n,m}^{(\delta)}
=\Delta^m x_{n,m}^{(\delta)} \sum_{N_m=1}^{n}{\cdots \sum_{N_2=1}^{N_3}{\sum_{N_1=1}^{N_2}{1}}}
=\delta \sum_{N_m=1}^{n}{\cdots \sum_{N_2=1}^{N_3}{\sum_{N_1=1}^{N_2}{1}}}
.$$
\end{theorem}
\begin{proof}
A particular case of Lemma \ref{l 3.4} for $j=k=m$ is as follows, 
$$
\sum_{N_m=1}^{n}{\cdots \sum_{N_2=1}^{N_3}{\sum_{N_1=1}^{N_2}{\Delta^{m} x_{N_1,m}^{(\delta)}}}}
= x_{n,m}^{(\delta)}
.$$
By definition of a repeated sequence of degree $m$, $\Delta^m x_{n,m}^{(\delta)}$ is a constant. Hence, we can take it out of the summation. Doing so gives the desired theorem. 
\end{proof}
For repeated unity sequences, Theorem \ref{t 3.4} shows that these sequences are equal to repeated sums of ones. 
%Corollary 3.2
\begin{corollary} \label{c 3.1}
Any element of the $m$-th degree unity repeated sequence can be defined as an $m$-th order repeated sum of ones, 
$$
\hat{x}_{n,m}
=\sum_{N_m=1}^{n}{\cdots \sum_{N_2=1}^{N_3}{\sum_{N_1=1}^{N_2}{1}}}
.$$
\end{corollary}
\begin{proof}
Knowing that $\Delta^m \hat{x}_{n,m}=1$, this corollary is a direct consequence of Theorem \ref{t 3.4}. 
\end{proof}
\subsubsection{Explicit expression for a repeated sequence as a binomial coefficient}
We now prove that repeated sequences can be defined as binomial coefficients. 
%Theorem 3.5
\begin{theorem} \label{t 3.5}
For any $m,n \in \mathbb{N}$ where $n \geq 1$, we have that 
$$
x_{n,m}^{(\delta)}
=\Delta^{m}x_{n,m}^{(\delta)} \binom{n+m-1}{m}
=\delta \binom{n+m-1}{m}
.$$
\end{theorem}
\begin{proof}
By applying Corollary \ref{c 2.3} to Theorem \ref{t 3.4}, we obtain the theorem. 
\end{proof}
\begin{corollary} \label{c 3.2}
For any $m,n \in \mathbb{N}$ where $n \geq 1$, we have that 
$$
\hat{x}_{n,m}=\binom{n+m-1}{m}
.$$
\end{corollary}
\subsubsection{Sum of repeated sequences}
Now that the needed expressions for repeated sequences were developed, one can proceed to develop formulae for the sum of the terms of such sequences. 
%Theorem 3.6
\begin{theorem} \label{t 3.6}
For any $k \in \mathbb{N^*}$ and for any $m,n \in \mathbb{N}$ where $n \geq 1$, we have that 
\begin{equation*} 
\sum_{N_k=1}^{n}{\cdots \sum_{N_1=1}^{N_2}{x_{N_1,m}^{(\delta)}}}
=\Delta^{m} x_{n,m}^{(\delta)}\sum_{N_{k+m}=1}^{n}{\cdots \sum_{N_1=1}^{N_2}{1}}
=\Delta^{m} x_{n,m}^{(\delta)}\binom{n+m+k-1}{m+k}
=x_{n,m+k}^{(\delta)}
.\end{equation*}
\end{theorem}
\begin{proof}
Applying Theorem \ref{t 3.4}, we get 
%$$
%x_{n,m}^{(\delta)}
%=\Delta^m x_{n,m}^{(\delta)} \sum_{j_m=1}^{n}{\cdots \sum_{j_2=1}^{j_3}{\sum_{j_1=1}^{j_2}{1}}}
%.$$
%Hence, 
\begin{equation*}
\sum_{N_k=1}^{n}{\cdots \sum_{N_1=1}^{N_2}{x_{N_1,m}^{(\delta)}}}
=\sum_{N_k=1}^{n}{\cdots \sum_{N_1=1}^{N_2}{ \sum_{j_m=1}^{N_1}{\cdots \sum_{j_2=1}^{j_3}{\sum_{j_1=1}^{j_2}{\Delta^m x_{N_1,m}^{(\delta)}}}}}}
.\end{equation*}
Knowing that $\forall i,j \in \mathbb{N}^*, \,\,\,\, \Delta^m x_{i,m}^{(\delta)}=\Delta^m x_{j,m}^{(\delta)}$, hence, $\Delta^m x_{N_1,m}^{(\delta)}=\Delta^m x_{n,m}^{(\delta)}$. 
\begin{equation*}
\sum_{N_k=1}^{n}{\cdots \sum_{N_1=1}^{N_2}{x_{N_1,m}^{(\delta)}}}
=\Delta^m x_{n,m}^{(\delta)}\sum_{N_k=1}^{n}{\cdots \sum_{N_1=1}^{N_2}{ \sum_{j_m=1}^{N_1}{\cdots \sum_{j_2=1}^{j_3}{\sum_{j_1=1}^{j_2}{1}}}}}
=\Delta^m x_{n,m}^{(\delta)}\sum_{N_{k+m}=1}^{n}{\cdots \sum_{N_1=1}^{N_2}{1}}
.\end{equation*}
The first equality is proven.
The second equality is obtained by applying Theorem \ref{t 2.1}. \\
%By applying Theorem \ref{t 3.4},  
%$$
%x_{n,m+k}^{(\delta)}
%=\Delta^m x_{n,m+k}^{(\delta)} \sum_{N_{m+k}=1}^{n}{\cdots \sum_{N_2=1}^{N_3}{\sum_{N_1=1}^{N_2}{1}}}
%.$$
Knowing that $\Delta^{m+k} x_{n,m+k}^{(\delta)}=\Delta^{m} x_{n,m}^{(\delta)}=\delta$, we apply Theorem \ref{t 3.4} to get 
$$
x_{n,m+k}^{(\delta)}
=\Delta^m x_{n,m+k}^{(\delta)} \sum_{N_{m+k}=1}^{n}{\cdots \sum_{N_2=1}^{N_3}{\sum_{N_1=1}^{N_2}{1}}}
=\Delta^m x_{n,m}^{(\delta)} \sum_{N_{m+k}=1}^{n}{\cdots \sum_{N_2=1}^{N_3}{\sum_{N_1=1}^{N_2}{1}}}
.$$
This completes the proof of the third equality.  
\end{proof}
%Corollary 3.3
\begin{corollary} \label{c 3.3}
For any $k \in \mathbb{N^*}$ and for any $m,n \in \mathbb{N}$ where $n \geq 1$, we have that 
\begin{equation*}
\sum_{N_k=1}^{n}{\cdots \sum_{N_1=1}^{N_2}{\hat{x}_{N_1,m}}}
=\sum_{N_{k+m}=1}^{n}{\cdots \sum_{N_1=1}^{N_2}{1}}
=\binom{n+m+k-1}{m+k}
=\hat{x}_{n,m+k}
.\end{equation*}
\end{corollary}
\begin{proof}
Applying Theorem \ref{t 3.6} and noting that $\Delta^m \hat{x}_{n,m}=1$, we obtain this corollary. 
\end{proof}
\begin{example}
{\em The sequence $(\hat{x}_{n,1})$ is a sequence of terms where the difference of two consecutive terms is 1: $1,2,3,4,5, \ldots$ . Using Corollary \ref{c 3.3}, we calculate the repeated sum of order $3$ of its first $10$ terms as follows: }
$$
\sum_{N_3=1}^{10}{\sum_{N_2=1}^{N_3}{\sum_{N_1=1}^{N_2}{\hat{x}_{N_1,1}}}}
=\binom{10+1+3-1}{1+3}=\binom{13}{4}=715
.$$
\end{example}
\begin{example}
{\em The sequence $(\hat{x}_{n,2})$ is a sequence of terms where the difference of two consecutive terms is increasing by 1: $1,3,6,10,15, \ldots$ . Using Corollary \ref{c 3.3}, we calculate the repeated sum of order $3$ of its first $10$ terms as follows: }
$$
\sum_{N_3=1}^{10}{\sum_{N_2=1}^{N_3}{\sum_{N_1=1}^{N_2}{\hat{x}_{N_1,2}}}}
=\binom{10+2+3-1}{2+3}=\binom{14}{5}=2002
.$$
\end{example}
\section{Reduction of Binomial-Harmonic sums}
\label{Reduction of Binomial-Harmonic Sums}
The harmonic sum has been studied independently by Oresme \cite{Oresme}, Mengoli \cite{Mengoli}, Johann Bernoulli \cite{JohannBernoulli}, Jacob Bernoulli \cite{JacobBernoulli1,JacobBernoulli2}, and most importantly by Euler. A particularly interesting variant of this sum is that combining binomial coefficients with harmonic sums. We will refer to such sums as {\it binomial-harmonic sums}. In this section, we develop various formulas for the reduction and calculation of such sums. We also present a formula relating the repeated harmonic sum to the shifted harmonic sum. 
\subsection{Binomial-Harmonic sum}
We begin by proving the following reduction formula to simplify Binomial-Harmonic sums. 
%Theorem 4.1
\begin{theorem} \label{t 4.1}
For any $m,p,n \in \mathbb{N}$ such that $n \geq p$, we have that 
$$
\sum_{N=p}^{n}{\binom{N+m}{m}\sum_{i=1+m}^{N+m}{\frac{1}{i}}}
=\binom{n+m+1}{m+1}\sum_{i=2+m}^{n+m+1}{\frac{1}{i}}
-\binom{p+m}{m+1}\sum_{i=2+m}^{p+m}{\frac{1}{i}}
.$$
\end{theorem}
\begin{proof}
1. Base case: verify true for $n=p$. 
\begin{equation*}
\begin{split}
&\binom{p+m+1}{m+1}\sum_{i=2+m}^{p+m+1}{\frac{1}{i}}
-\binom{p+m}{m+1}\sum_{i=2+m}^{p+m}{\frac{1}{i}} \\ 
&\,\,\,\,=\binom{p+m+1}{m+1} \frac{1}{p+m+1}+\left[\binom{p+m+1}{m+1}-\binom{p+m}{m+1}\right]\sum_{i=2+m}^{p+m}{\frac{1}{i}}
.\end{split}
\end{equation*}
From the Pascal's Triangle identity, $\binom{p+m+1}{m+1}-\binom{p+m}{m+1}=\binom{p+m}{m}$. 
In addiction, 
$$
\binom{p+m+1}{m+1} \frac{1}{p+m+1}
=\binom{p+m}{m} \frac{1}{m+1}
.$$
Hence, substituting back, we get 
\begin{equation*}
\binom{p+m+1}{m+1}\sum_{i=2+m}^{p+m+1}{\frac{1}{i}}
-\binom{p+m}{m+1}\sum_{i=2+m}^{p+m}{\frac{1}{i}}
=\binom{p+m}{m}\sum_{i=1+m}^{p+m}{\frac{1}{i}}
=\sum_{N=p}^{p}{\binom{N+m}{m}\sum_{i=1+m}^{N+m}{\frac{1}{i}}}
.\end{equation*}
2. Induction hypothesis: assume the statement is true until $n$. 
$$
\sum_{N=p}^{n}{\binom{N+m}{m}\sum_{i=1+m}^{N+m}{\frac{1}{i}}}
=\binom{n+m+1}{m+1}\sum_{i=2+m}^{n+m+1}{\frac{1}{i}}
-\binom{p+m}{m+1}\sum_{i=2+m}^{p+m}{\frac{1}{i}}
.$$
3. Induction step: we will show that this statement is true for $(n+1)$. \\
We have to show the following statement to be true:  
$$
\sum_{N=p}^{n+1}{\binom{N+m}{m}\sum_{i=1+m}^{N+m}{\frac{1}{i}}}
=\binom{n+m+2}{m+1}\sum_{i=2+m}^{n+m+2}{\frac{1}{i}}
-\binom{p+m}{m+1}\sum_{i=2+m}^{p+m}{\frac{1}{i}}
.$$
$$ \\ $$ 
$$
\sum_{N=p}^{n+1}{\binom{N+m}{m}\sum_{i=1+m}^{N+m}{\frac{1}{i}}}
=\sum_{N=p}^{n}{\binom{N+m}{m}\sum_{i=1+m}^{N+m}{\frac{1}{i}}}
+\binom{n+m+1}{m}\sum_{i=1+m}^{n+m+1}{\frac{1}{i}}
.$$
From the induction hypothesis, 
\begin{equation*}
\begin{split}
\sum_{N=p}^{n+1}{\binom{N+m}{m}\sum_{i=1+m}^{N+m}{\frac{1}{i}}}
&=\binom{n+m+1}{m+1}\sum_{i=2+m}^{n+m+1}{\frac{1}{i}}
-\binom{p+m}{m+1}\sum_{i=2+m}^{p+m}{\frac{1}{i}}
+\binom{n+m+1}{m}\sum_{i=1+m}^{n+m+1}{\frac{1}{i}} \\
&=\left[\binom{n+m+1}{m+1}+\binom{n+m+1}{m}\right]\sum_{i=2+m}^{n+m+1}{\frac{1}{i}}
+\binom{n+m+1}{m}\frac{1}{m+1} \\ 
&\,\,\,\,\,\,-\binom{p+m}{m+1}\sum_{i=2+m}^{p+m}{\frac{1}{i}}
.\end{split}
\end{equation*}
From the Pascal's Triangle identity, $\binom{n+m+1}{m+1}+\binom{n+m+1}{m}=\binom{n+m+2}{m+1}$.
In addition, 
\begin{equation*}
\begin{split}
\resizebox{\linewidth}{!}{$\displaystyle{
\binom{n+m+1}{m}\frac{1}{m+1}
=\frac{(n+m+1)!}{m!(n+1)!}\frac{1}{m+1}
=\frac{(n+m+2)!}{(m+1)!(n+1)!}\frac{1}{n+m+2} 
=\binom{n+m+2}{m+1}\frac{1}{n+m+2}.}$}
\end{split}
\end{equation*}
Hence, 
\begin{equation*}
\begin{split}
\sum_{N=p}^{n+1}{\binom{N+m}{m}\sum_{i=1+m}^{N+m}{\frac{1}{i}}}
&=\binom{n+m+2}{m+1}\sum_{i=2+m}^{n+m+2}{\frac{1}{i}}
-\binom{p+m}{m+1}\sum_{i=2+m}^{p+m}{\frac{1}{i}}
.\end{split}
\end{equation*}
%\begin{equation*}
%\begin{split}
%\sum_{N=p}^{n}{\binom{N+m}{m}\sum_{i=1+m}^{N+m}{\frac{1}{i}}}
%&=\sum_{N=1}^{n}{\binom{N+m}{m}\sum_{i=1+m}^{N+m}{\frac{1}{i}}}
%-\sum_{N=1}^{p-1}{\binom{N+m}{m}\sum_{i=1+m}^{N+m}{\frac{1}{i}}}\\
%&=\sum_{N=1}^{n}{\binom{N+m}{m}\sum_{i=1}^{N}{\frac{1}{i+m}}}
%-\sum_{N=1}^{p-1}{\binom{N+m}{m}\sum_{i=1}^{N}{\frac{1}{i+m}}}.
%\end{split}
%\end{equation*}
%We apply Theorem 3.1 from \cite{RecurrentSums} to each part individually, 
%$$
%\sum_{N=1}^{n}{\binom{N+m}{m}\sum_{i=1}^{N}{\frac{1}{i+m}}}
%=\sum_{i=1}^{n}{\frac{1}{i+m}\sum_{N=i}^{n}{\binom{N+m}{m}}}
%=\sum_{i=1}^{n}{\frac{1}{i+m}\sum_{N=i+m}^{n+m}{\binom{N}{m}}},
%$$
%$$
%\sum_{N=1}^{p-1}{\binom{N+m}{m}\sum_{i=1}^{N}{\frac{1}{i+m}}}
%=\sum_{i=1}^{p-1}{\frac{1}{i+m}\sum_{N=i}^{n}{\binom{N+m}{m}}}
%=\sum_{i=1}^{p-1}{\frac{1}{i+m}\sum_{N=i+m}^{n+m}{\binom{N}{m}}}.
%$$
\end{proof}
%Corollary 4.1
\begin{corollary} \label{c 4.1}
For $p=1$, Theorem \ref{t 4.1} becomes 
$$
\sum_{N=1}^{n}{\binom{N+m}{m}\sum_{i=1+m}^{N+m}{\frac{1}{i}}}
=\binom{n+m+1}{m+1}\sum_{i=2+m}^{n+m+1}{\frac{1}{i}}
.$$
\end{corollary}
\subsection{Repeated Binomial-Harmonic sum}
Now we generalize Theorem \ref{t 4.1} and develop a formula for simplifying repeated Binomial-Harmonic sums. 
%Theorem 4.2
\begin{theorem} \label{t 4.2}
For any $k \in \mathbb{N^*}$ and for any $m,p,n \in \mathbb{N}$ such that $n \geq p$, we have that 
\begin{equation*}
\begin{split}
&\sum_{N_{k}=p}^{n}{\cdots \sum_{N_{1}=p}^{N_2}{\left[\binom{N_1+m}{m} \sum_{i=1+m}^{N_1+m}{\frac{1}{i}}\right]}} \\
&\,\,\,\,=\binom{n+m+k}{m+k}\sum_{i=1+m+k}^{n+m+k}{\frac{1}{i}}-\sum_{j=0}^{k-1}{\binom{n-p+j}{j}\binom{p-1+m+k-j}{m+k-j} \sum_{i=1+m+k-j}^{p-1+m+k-j}{\frac{1}{i}}}
.\end{split}
\end{equation*}
\end{theorem}
\begin{proof}
1. Base case: verify true for $k=1$. \\ 
For $k=1$ this theorem will reduce to Theorem \ref{t 4.1}. Hence, the case for $k=1$ is proven. \\ 
2. Induction hypothesis: assume the statement is true until $k$.  
\begin{equation*}
\begin{split}
&\sum_{N_{k}=p}^{n}{\cdots \sum_{N_{1}=p}^{N_2}{\left[\binom{N_1+m}{m} \sum_{i=1+m}^{N_1+m}{\frac{1}{i}}\right]}} \\
&\,\,\,\,=\binom{n+m+k}{m+k}\sum_{i=1+m+k}^{n+m+k}{\frac{1}{i}}-\sum_{j=0}^{k-1}{\binom{n-p+j}{j}\binom{p-1+m+k-j}{m+k-j} \sum_{i=1+m+k-j}^{p-1+m+k-j}{\frac{1}{i}}}
.\end{split}
\end{equation*}
3. Induction step: we will show that this statement is true for $(k+1)$. \\ 
We have to show the following statement to be true:  
\begin{equation*}
\begin{split}
&\sum_{N_{k+1}=p}^{n}{\cdots \sum_{N_{1}=p}^{N_2}{\left[\binom{N_1+m}{m} \sum_{i=1+m}^{N_1+m}{\frac{1}{i}}\right]}} \\
&\,\,\,\,=\binom{n+m+k+1}{m+k+1}\sum_{i=2+m+k}^{n+m+k+1}{\frac{1}{i}}-\sum_{j=0}^{k}{\binom{n-p+j}{j}\binom{p+m+k-j}{m+k-j+1} \sum_{i=2+m+k-j}^{p+m+k-j}{\frac{1}{i}}}
.\end{split}
\end{equation*}
Let us denote the first term of the equality as B. 
\begin{equation*}
%\sum_{N_{k+1}=p}^{n}{\cdots \sum_{N_{1}=p}^{N_2}{\left[\binom{N_1+m}{m} \sum_{i=1+m}^{N_1+m}{\frac{1}{i}}\right]}}
B=\sum_{N_{k+1}=p}^{n}{\left \{\sum_{N_{k}=p}^{N_{k+1}}{\cdots \sum_{N_{1}=p}^{N_2}{\left[\binom{N_1+m}{m} \sum_{i=1+m}^{N_1+m}{\frac{1}{i}}\right]}}\right \}}
.\end{equation*}
Applying the induction hypothesis, we get 
\begin{equation*}
\begin{split}
B
&=\sum_{N_{k+1}=p}^{n}{\binom{N_{k+1}+m+k}{m+k}\sum_{i=1+m+k}^{N_{k+1}+m+k}{\frac{1}{i}}} \\
&-\sum_{N_{k+1}=p}^{n}{\sum_{j=0}^{k-1}{\binom{N_{k+1}-p+j}{j}\binom{p-1+m+k-j}{m+k-j} \sum_{i=1+m+k-j}^{p-1+m+k-j}{\frac{1}{i}}}} \\
&=\sum_{N_{k+1}=p}^{n}{\binom{N_{k+1}+m+k}{m+k}\sum_{i=1+m+k}^{N_{k+1}+m+k}{\frac{1}{i}}} \\
&-\sum_{j=0}^{k-1}{\left \{\sum_{N_{k+1}=p}^{n}{\binom{N_{k+1}-p+j}{j}}\right \}\binom{p-1+m+k-j}{m+k-j} \sum_{i=1+m+k-j}^{p-1+m+k-j}{\frac{1}{i}}}
.\end{split}
\end{equation*}
Applying Theorem \ref{t 4.1} to the first term and applying Theorem \ref{t 2.1} to the term in curly brackets,    
\begin{equation*}
\begin{split}
B
&=\binom{n+m+k+1}{m+k+1}\sum_{i=2+m+k}^{n+m+k+1}{\frac{1}{i}}
-\binom{p+m+k}{m+k+1}\sum_{i=2+m+k}^{p+m+k}{\frac{1}{i}} \\
&-\sum_{j=0}^{k-1}{\binom{n-p+j+1}{j+1}\binom{p-1+m+k-j}{m+k-j} \sum_{i=1+m+k-j}^{p-1+m+k-j}{\frac{1}{i}}} \\
&=\binom{n+m+k+1}{m+k+1}\sum_{i=2+m+k}^{n+m+k+1}{\frac{1}{i}}
-\binom{p+m+k}{m+k+1}\sum_{i=2+m+k}^{p+m+k}{\frac{1}{i}} \\
&-\sum_{j=1}^{k}{\binom{n-p+j}{j}\binom{p+m+k-j}{m+k-j+1} \sum_{i=2+m+k-j}^{p+m+k-j}{\frac{1}{i}}} \\
&=\binom{n+m+k+1}{m+k+1}\sum_{i=2+m+k}^{n+m+k+1}{\frac{1}{i}}
-\sum_{j=0}^{k}{\binom{n-p+j+1}{j+1}\binom{p-1+m+k-j}{m+k-j} \sum_{i=1+m+k-j}^{p-1+m+k-j}{\frac{1}{i}}}
.\end{split}
\end{equation*}
Hence, the theorem holds true for $(k+1)$. Thus, the theorem is proven by induction. 
\end{proof}
%Corollary 4.2
\begin{corollary} \label{c 4.2}
For $p=1$, Theorem \ref{t 4.2} becomes 
$$
\sum_{N_{k}=1}^{n}{\cdots \sum_{N_{1}=1}^{N_2}{\left[\binom{N_1+m}{m} \sum_{i=1+m}^{N_1+m}{\frac{1}{i}}\right]}}
=\binom{n+m+k}{m+k}\sum_{i=1+m+k}^{n+m+k}{\frac{1}{i}}
.$$
\end{corollary}
A particular case of Corollary \ref{c 4.2} which is of major interest is the repeated harmonic sum. 
%Corollary 4.3
\begin{corollary} \label{c 4.3}
The repeated harmonic sum of order $(m+1)$ is related to the shifted harmonic sum of order $1$ by the following relation, 
$$
\sum_{N_{m+1}=1}^{n}{\cdots \sum_{N_2=1}^{N_3}{\sum_{N_1=1}^{N_2}{\frac{1}{N_1}}}}
=\binom{n+m}{m}\sum_{N=1+m}^{n+m}{\frac{1}{N}}
.$$
\end{corollary}
\begin{proof}
Applying Corollary \ref{c 4.2} with $m=0$, we get this corollary. 
%$$
%\sum_{N_{k}=1}^{n}{\cdots \sum_{N_{1}=1}^{N_2}{\left[\binom{N_1}{0} \sum_{i=1}^{N_1}{\frac{1}{i}}\right]}}
%=\sum_{N_{k+1}=1}^{n}{\cdots \sum_{N_{1}=1}^{N_2}{\frac{1}{N_1}}}
%=\binom{n+k}{k}\sum_{N=1+k}^{n+k}{\frac{1}{N}}
%.$$
\end{proof}
\section{Application to repeated sums}\label{Application to the Calculation of Repeated Sums}

\subsection{Reduction of a repeated sum}
Let us denote the repeated sum of order $m$ of the sequence $a_N$  with lower and upper bounds respectively $q$ and $n$ as $S_{m,q,n}(a_N)$. For simplicity, we will denote it as  $S_{m,q,n}$. 
$$
S_{m,q,n}=\sum_{N_m=q}^{n}{\sum_{N_{m-1}=q}^{N_m}{\cdots \sum_{N_1=q}^{N_2}{a_{N_1}}}}.
$$
Recursively, it can be defined by:
\begin{align*}
\begin{cases}
S_{m,q,n}=\sum_{N_m=q}^{n}{S_{m-1,q,N_m}},\\
S_{0,q,n}\,\,=a_n \,\,\,\, \forall q,n \in \mathbb{N}.
\end{cases}
\end{align*}
\subsubsection{Variation Formula}
In this section, we will present formulas to express the variation of a repeated sum of order $m$ ($S_{m,q,n+1}-S_{m,q,n}$) in terms of lower order repeated sums. Equivalently, these formulas can be used to express $S_{m,q,n+1}$ in terms of $S_{m,q,n}$ and lower order repeated sums. 

We begin by presenting the variation formula that allows $S_{m,q,n+1}$ to be expressed in terms of $S_{m,q,n}$ and of repeated sums of order going from $0$ to $(m-1)$. 
\begin{theorem}
For any $n,q\in \mathbb{N}$ such that $n\geq q$, and for any $m\in\mathbb{N^*}$, we have 
$$
\sum_{N_m=q}^{n+1}{\sum_{N_{m-1}=q}^{N_m}{\cdots \sum_{N_1=q}^{N_2}{a_{N_1}}}}
=\sum_{k=1}^{m}{\left( \sum_{N_k=q}^{n}{\sum_{N_{k-1}=q}^{N_k}{\cdots \sum_{N_1=q}^{N_2}{a_{N_1}}}}\right)}+a_{n+1}.
$$
Or using the notation, 
$$
S_{m,q,n+1}
=\sum_{k=1}^{m}{S_{k,q,n}}+S_{0,q,n+1}.
$$
\end{theorem}
\begin{proof}
In \cite{RecurrentSums}, the author presented the following variation formula:
$$\sum_{N_m=q}^{n+1}{\cdots \sum_{N_1=q}^{N_2}{a_{(m);N_m}\cdots a_{(1);N_1}}}
=\sum_{k=0}^{m}{\left(\prod_{j=0}^{m-k-1}{a_{(m-j);n+1}}\right)\left(\sum_{N_k=q}^{n}{\cdots \sum_{N_1=q}^{N_2}{a_{(k);N_k}\cdots a_{(1);N_1}}}\right)}.$$
Setting $a_{(m);N}=\cdots = a_{(2);N}=1$ and $a_{(1);N}=a_{N}$, we obtain this theorem. 
\end{proof}
The variation of a repeated sum can also be expressed in terms of only a certain range of lower order repeated sums. In other words, $S_{m,q,n+1}$ can be expressed in terms of $S_{m,q,n}$ and of repeated sums of order going only from $p$ to $(m-1)$. To do so, we give the following theorem.  
\begin{theorem}
For any $n,q\in \mathbb{N}$ such that $n\geq q$, and for any $m, p\in\mathbb{N^*}$ such that $p \leq m$, we have 
$$
\sum_{N_m=q}^{n+1}{\sum_{N_{m-1}=q}^{N_m}{\cdots \sum_{N_1=q}^{N_2}{a_{N_1}}}}
=\sum_{k=p+1}^{m}{\left( \sum_{N_k=q}^{n}{\sum_{N_{k-1}=q}^{N_k}{\cdots \sum_{N_1=q}^{N_2}{a_{N_1}}}}\right)}
+\sum_{N_p=q}^{n+1}{\sum_{N_{p-1}=q}^{N_p}{\cdots \sum_{N_1=q}^{N_2}{a_{N_1}}}}
.$$
Or using the notation, 
$$
S_{m,q,n+1}
=\sum_{k=p+1}^{m}{S_{k,q,n}}
+S_{p,q,n+1}
.$$
\end{theorem}
\begin{proof}
In \cite{RecurrentSums}, the author presented the following variation formula:
\begin{equation*}
\begin{split}
\sum_{N_m=q}^{n+1}{\cdots \sum_{N_1=q}^{N_2}{a_{(m);N_m}\cdots a_{(1);N_1}}}
=&\sum_{k=p+1}^{m}{\left(\prod_{j=0}^{m-k-1}{a_{(m-j);n+1}}\right)\left(\sum_{N_k=q}^{n}{\cdots \sum_{N_1=q}^{N_2}{a_{(k);N_k}\cdots a_{(1);N_1}}}\right)} \\
&+\left(\prod_{j=0}^{m-p-1}{a_{(m-j);n+1}}\right)\left(\sum_{N_p=q}^{n+1}{\cdots \sum_{N_1=q}^{N_2}{a_{(p);N_p}\cdots a_{(1);N_1}}}\right).
\end{split}
\end{equation*}
Setting $a_{(m);N}=\cdots = a_{(2);N}=1$ and $a_{(1);N}=a_{N}$, we obtain this theorem. 
\end{proof}
\subsubsection{Reduction Formula}
Now that we have proven the needed binomial coefficient definitions and identities, we derive a formula to simplify the repeated sum of any sequence $a_N$.  
%Theorem 5.1
\begin{theorem} \label{t 5.1}
For any $m\in \mathbb{N^*}$, for any $q,n \in \mathbb{N}$ such that $n \geq q$ and for any sequence $a_N$ defined in the interval $[q,n]$, we have that 
$$
\sum_{N_m=q}^{n}{\cdots \sum_{N_2=q}^{N_3}{\sum_{N_1=q}^{N_2}{a_{N_1}}}}
=\sum_{i=q}^{n}{\binom{n+(m-1)-i}{m-1}a_{i}}
.$$
\end{theorem}
\begin{proof}
Applying Theorem 3.3 from \cite{RecurrentSums}, for $a_{(m);N_m}= \cdots = a_{(2);N_2}=1$ and $a_{(1);N_1}=a_{N_1}$, we get 
$$
\sum_{N_m=q}^{n}{\cdots \sum_{N_2=q}^{N_3}{\sum_{N_1=q}^{N_2}{a_{N_1}}}}
=\sum_{N_1=q}^{n}{a_{N_1}\sum_{N_m=N_1}^{n}{\cdots \sum_{N_3=N_1}^{N_4}{\sum_{N_2=N_1}^{N_3}{1}}}}
.$$
By applying Corollary \ref{c 2.4} to the $(m-1)$ inner sums, we have 
$$
\sum_{N_m=N_1}^{n}{\cdots \sum_{N_3=N_1}^{N_4}{\sum_{N_2=N_1}^{N_3}{1}}}
=\binom{n-N_1+m-1}{m-1}
.$$
Hence, substituting back, we get the desired theorem. 
\end{proof}
Theorem \ref{t 5.1} offers a more computationally efficient way of calculating repeated sums as computing the repeated sum directly requires adding up $\binom{n-q+m}{m}$ terms while using the simplified form requires adding up only $n-q+1$ terms. Also notice that the number of terms added in the simplified form is independent of the order $m$. For example, let us consider a repeated sum of order $m=10$ and with lower and upper bounds respectively $q=1$ and $n=10$: computing this repeated sum directly requires adding $92 378$ terms while using the theorem requires adding only $10$ terms. 
\begin{remark}
{\em From Theorem \ref{t 5.1}, we can deduce that a term $a_i$ appears $\binom{n+(m-1)-i}{m-1}$ times in the repeated sum of the sequence $a_N$. }
\end{remark}
%Corollary 5.1
\begin{corollary} \label{c 5.1}
If the sums start at $1$, Theorem \ref{t 5.1} becomes 
$$
\sum_{N_m=1}^{n}{\cdots \sum_{N_2=1}^{N_3}{\sum_{N_1=1}^{N_2}{a_{N_1}}}}
=\sum_{i=1}^{n}{\binom{n+(m-1)-i}{m-1}a_{i}}
.$$
\end{corollary}
\subsubsection{Application to harmonic sums and repeated harmonic sums}
In this section, we apply Theorem \ref{t 5.1} to develop some formulae for harmonic sums and repeated harmonic sums. 

First, we introduce a formula for simplifying repeated harmonic sums.  
\begin{theorem} \label{t 5.2}
For any $m,n \in \mathbb{N^*}$, we have that 
\begin{equation*}
\sum_{N_m=1}^{n}{\cdots \sum_{N_2=1}^{N_3}{\sum_{N_1=1}^{N_2}{\frac{1}{N_1}}}}
=\sum_{i=1}^{n}{\binom{n+(m-1)-i}{m-1}\frac{1}{i}}
=\binom{n+m-1}{m-1}\sum_{i=m}^{n+m-1}{\frac{1}{i}}
.\end{equation*}
\end{theorem}
\begin{proof}
By equating the expression of the repeated harmonic sum obtained from Corollary \ref{c 4.3} with that obtained from Theorem \ref{t 5.1} for $a_N=\frac{1}{N}$, we obtain this theorem. 
%From Corollary \ref{c 4.3}, 
%\begin{equation*}
%\sum_{N_m=1}^{n}{\cdots \sum_{N_2=1}^{N_3}{\sum_{N_1=1}^{N_2}{\frac{1}{N_1}}}}
%=\binom{n+m-1}{m-1}\sum_{i=m}^{n+m-1}{\frac{1}{i}}
%.\end{equation*}
%By applying Theorem \ref{t 5.1} for $a_N=\frac{1}{N}$, 
%\begin{equation*}
%\sum_{N_m=1}^{n}{\cdots \sum_{N_2=1}^{N_3}{\sum_{N_1=1}^{N_2}{\frac{1}{N_1}}}}
%=\sum_{i=1}^{n}{\binom{n+(m-1)-i}{m-1}\frac{1}{i}}
%.\end{equation*}
%Hence, by equating the two expressions, we obtain the theorem. 
\end{proof}
Now we present identities related to the harmonic sum. We start with the following identity. 
\begin{theorem} \label{t 5.3}
For any $m,n \in \mathbb{N}$ such that $n \geq 1$, we have that 
\begin{equation*}
\sum_{i=1}^{n}{\frac{\binom{n}{i}}{\binom{n+m}{i}}\frac{1}{i}}
=\sum_{i=1}^{n}{\left[\prod_{k=0}^{i-1}{\frac{(n-k)}{(n+m-k)}}\right]\frac{1}{i}}
=\sum_{i=1+m}^{n+m}{\frac{1}{i}}
.\end{equation*}
\end{theorem}
\begin{proof}
From Theorem \ref{t 5.2} with $m$ substituted by $(m+1)$, 
\begin{equation*}
\sum_{i=1}^{n}{\binom{n+m-i}{m}\frac{1}{i}}
=\binom{n+m}{m}\sum_{i=1+m}^{n+m}{\frac{1}{i}}
.\end{equation*}
Knowing that 
\begin{equation*}
\frac{\binom{n+m-i}{m}}{\binom{n+m}{m}}
=\frac{\binom{n}{i}}{\binom{n+m}{i}}
=\prod_{k=0}^{i-1}{\frac{(n-k)}{(n+m-k)}}
,\end{equation*}
hence, substituting back, we obtain the theorem. 
%By expressing the binomial coefficients using the factorial definition and simplifying, we get 
%\begin{equation*}
%\sum_{i=1}^{n}{\frac{(n+m-i)!}{(n-i)!}\frac{1}{i}}
%=\frac{(n+m)!}{n!}\sum_{i=1+m}^{n+m}{\frac{1}{i}}
%.\end{equation*}
%Multiplying by $\frac{n!}{(n+m)!}$, we get 
%\begin{equation*}
%\sum_{i=1}^{n}{\frac{(n+m-i)!}{(n+m)!}\frac{n!}{(n-i)!}\frac{1}{i}}
%=\sum_{i=1+m}^{n+m}{\frac{1}{i}}
%.\end{equation*}
%Knowing that 
%\begin{equation*}
%\frac{(n+m-i)!}{(n+m)!}\frac{n!}{(n-i)!}
%=\frac{\binom{n}{i}}{\binom{n+m}{i}}
%=\prod_{k=0}^{i-1}{\frac{(n-k)}{(n+m-k)}}
%,\end{equation*}
%hence, substituting back, we obtain the theorem. 
\end{proof}
In order to prove the second identity, we need to first prove the following two lemmas. 
% Lemma 5.1 
\begin{lemma} \label{l 5.1}
For any $k \in \mathbb{N}^*$, we have that 
\begin{equation*}
\sum_{i=1}^{2k}{\frac{1}{i}}
=(2k+1)\sum_{i=1}^{k}{\frac{1}{(2k+1-i)i}}
=\frac{(2k+1)}{2}\sum_{i=1}^{2k}{\frac{1}{(2k+1-i)i}}
.\end{equation*}
\end{lemma}
\begin{proof}
We split the sum, modify its parts, then recombine them to obtain, 
\begin{equation*}
\sum_{i=1}^{2k}{\frac{1}{i}}
=\sum_{i=1}^{k}{\frac{1}{i}}
+\sum_{i=k+1}^{2k}{\frac{1}{i}}
=\sum_{i=1}^{k}{\frac{1}{i}}
+\sum_{i=1}^{k}{\frac{1}{2k+1-i}}
=(2k+1)\sum_{i=1}^{k}{\frac{1}{(2k+1-i)i}}.
\end{equation*}
Similarly, 
\begin{equation*}
\begin{split}
\frac{(2k+1)}{2}\sum_{i=1}^{2k}{\frac{1}{(2k+1-i)i}}
&=\frac{(2k+1)}{2}\left[\sum_{i=1}^{k}{\frac{1}{(2k+1-i)i}}
+\sum_{i=k+1}^{2k}{\frac{1}{(2k+1-i)i}} \right] \\
&=\frac{(2k+1)}{2}\left[\sum_{i=1}^{k}{\frac{1}{(2k+1-i)i}}
+\sum_{i=1}^{k}{\frac{1}{(2k+1-i)i}} \right] \\
&=(2k+1)\sum_{i=1}^{k}{\frac{1}{(2k+1-i)i}}.
\end{split}
\end{equation*}
Equating, we obtain the lemma. 
\end{proof}
% Lemma 5.2 
\begin{lemma} \label{l 5.2}
For any $k \in \mathbb{N}^*$, we have that 
\begin{equation*}
\sum_{i=1}^{2k-1}{\frac{1}{i}}
=(2k)\sum_{i=1}^{k}{\frac{1}{(2k-i)i}}-\frac{1}{k}
=(2k)\sum_{i=1}^{k-1}{\frac{1}{(2k-i)i}}+\frac{1}{k}
=(k)\sum_{i=1}^{2k-1}{\frac{1}{(2k-i)i}}
.\end{equation*}
\end{lemma}
\begin{proof}
We split the sum, modify its parts, then recombine them to obtain, 
\begin{equation*}
\sum_{i=1}^{2k-1}{\frac{1}{i}}
=\sum_{i=1}^{k}{\frac{1}{i}}
+\sum_{i=k}^{2k-1}{\frac{1}{i}}-\frac{1}{k}
=\sum_{i=1}^{k}{\frac{1}{i}}
+\sum_{i=1}^{k}{\frac{1}{2k-i}}-\frac{1}{k}
=(2k)\sum_{i=1}^{k}{\frac{1}{(2k-i)i}}-\frac{1}{k}.
\end{equation*}
Similarly, 
\begin{equation*}
\sum_{i=1}^{2k-1}{\frac{1}{i}}
=\sum_{i=1}^{k-1}{\frac{1}{i}}
+\sum_{i=k+1}^{2k-1}{\frac{1}{i}}+\frac{1}{k}
=\sum_{i=1}^{k-1}{\frac{1}{i}}
+\sum_{i=1}^{k-1}{\frac{1}{2k-i}}+\frac{1}{k}
=(2k)\sum_{i=1}^{k-1}{\frac{1}{(2k-i)i}}+\frac{1}{k}.
\end{equation*}
Finally, to obtain the last part of this lemma, we perform the following, 
\begin{equation*}
\begin{split}
(k)\sum_{i=1}^{2k-1}{\frac{1}{(2k-i)i}}
&=(k)\left[\sum_{i=1}^{k-1}{\frac{1}{(2k-i)i}}
+\frac{1}{k^2}+\sum_{i=k+1}^{2k-1}{\frac{1}{(2k-i)i}} \right] \\
&=(k)\left[\sum_{i=1}^{k-1}{\frac{1}{(2k-i)i}}
+\sum_{i=1}^{k-1}{\frac{1}{(2k-i)i}} \right] +\frac{1}{k} \\
&=(2k)\sum_{i=1}^{k-1}{\frac{1}{(2k-i)i}}+\frac{1}{k}.
\end{split}
\end{equation*}
Equating, we obtain the lemma. 
\end{proof}
Now that both needed lemmas have been proven, we give the second identity. 
\begin{theorem} \label{t 5.4}
For any $n \in \mathbb{N^\ast}$, we have that 
\begin{equation*}
\frac{n+1}{2}\sum_{i=1}^{n}{\frac{1}{(n+1-i)i}}
=\sum_{i=1}^{n}{\frac{1}{i}}
.\end{equation*}
\end{theorem}
\begin{proof}
For this proof, a proof by case will be performed. Thus, the proof will be divided into proving two complementary cases: the case when $n$ is even and the case when $n$ is odd. 
\begin{itemize}
\item For $n=2k$ (where $k \in \mathbb{N^*}$): \\
This case corresponds to Lemma \ref{l 5.1} that we have already proven. 
\item For $n=2k-1$ (where $k \in \mathbb{N^*}$): \\
This case corresponds to Lemma \ref{l 5.2} that we have already proven.
\end{itemize}
The formula holds true for both even and odd values of $n$, hence, it holds for any $n\geq 1$.  
\end{proof}
\subsection{Reduction of the repeated ``Binomial-Sequence" sum}
In this section, we prove a generalization of Theorem \ref{t 5.1}. 
This generalization is illustrated in the following theorem. The following theorem simplifies the repeated ``Binomial-Sequence" sum.
% Theorem 5.6 
\begin{theorem} \label{t 5.6}
For any $k \in \mathbb{N^*}$, for any $m,q,n \in \mathbb{N}$ where $n \geq q$ and for any sequence $a_N$ defined in the interval $[q,n]$, we have that 
\begin{equation*}
\sum_{N_k=q}^{n}{\cdots \sum_{N_2=q}^{N_3}{\sum_{N_1=q}^{N_2}{\binom{(N_2-N_1)+m}{m}a_{N_1}}}}
=\sum_{N=q}^{n}{\binom{(n-N)+m+k-1}{m+k-1}a_{N}}
.\end{equation*}
\end{theorem}
\begin{proof}
By applying Theorem \ref{t 5.1} to the inner sum, we have 
\begin{equation*}
\begin{split}
\sum_{N_k=q}^{n}{\cdots \sum_{N_2=q}^{N_3}{\sum_{N_1=q}^{N_2}{\binom{(N_2-N_1)+m}{m}a_{N_1}}}}
&=\sum_{N_k=q}^{n}{\cdots \sum_{N_2=q}^{N_3}{\sum_{j_{m+1}=q}^{N_2}{\cdots \sum_{j_{1}=q}^{j_2}{a_{j_1}}}}}\\
&=\sum_{i_{m+k}=q}^{n}{\cdots \sum_{i_2=q}^{i_3}{\sum_{i_1=q}^{i_2}{a_{i_1}}}}\\
&=\sum_{N=q}^{n}{\binom{(n-N)+m+k-1}{m+k-1}a_{N}}
\end{split}
\end{equation*}
\end{proof}
% Corollary 5.3 
\begin{corollary} \label{c 5.3}
If the summations start at $1$, Theorem \ref{t 5.6} becomes, 
\begin{equation*}
\sum_{N_k=1}^{n}{\cdots \sum_{N_1=1}^{N_2}{\binom{(N_2-N_1)+m}{m}a_{N_1}}}
=\sum_{N=1}^{n}{\binom{(n-N)+m+k-1}{m+k-1}a_{N}}.
\end{equation*}
\end{corollary}
For the simple ``Binomial-Sequence" sum, Theorem \ref{t 5.6} and Corollary \ref{c 5.3} reduce to the following. 
\begin{theorem} \label{t 5.5}
For any $m \in \mathbb{N}^{*}$, for any $q,n\in \mathbb{N}$ where $n \geq q$ and for any sequence $a_N$ defined in the interval $[q,n]$, we have that 
$$
\sum_{N_2=q}^{n}{\sum_{N_1=q}^{N_2}{\binom{(N_2-N_1)+(m-1)}{m-1}a_{N_1}}}
=\sum_{N=q}^{n}{\binom{(n-N)+m}{m}a_{N}}
.$$
\end{theorem}
% Corollary 5.2 
\begin{corollary} \label{c 5.2}
If the summations start at $1$, Theorem \ref{t 5.5} becomes, 
$$
\sum_{N_2=1}^{n}{\sum_{N_1=1}^{N_2}{\binom{(N_2-N_1)+(m-1)}{m-1}a_{N_1}}}
=\sum_{N=1}^{n}{\binom{(n-N)+m}{m}a_{N}}
.$$
\end{corollary}
%\subsection{Inversion Formula}
%%%%%%%%%%%%%%%%%%%%%%%%%%%%%%%%%%%%%
\bibliographystyle{elsarticle-num}
\bibliography{Repeated_Sums_2}

\begin{thebibliography}{10}
\expandafter\ifx\csname url\endcsname\relax
  \def\url#1{\texttt{#1}}\fi
\expandafter\ifx\csname urlprefix\endcsname\relax\def\urlprefix{URL }\fi
\expandafter\ifx\csname href\endcsname\relax
  \def\href#1#2{#2} \def\path#1{#1}\fi

\bibitem{RecurrentSums}
R.~El~Haddad, Recurrent sums and partition identities, arXiv preprint
  arXiv:2101.09089 (2021).

\bibitem{MultipleSums}
R.~El~Haddad, Multiple sums and partition identities, arXiv preprint
  arXiv:2102.00821 (2021).

\bibitem{pascal1978traite}
B.~Pascal, Trait{\'e} du triangle arithm{\'e}tique avec quelques autres petits
  traitez sur la mesme mati{\`e}re, Chez Gvillavme Desprer, 1978.

\bibitem{pascal2011traite}
B.~Pascal, Trait{\'e} du triangle arithm{\'e}tique, {\'E}d. les Caract{\`e}res
  d'Ulysse, 2011.

\bibitem{whiteside1961newton}
D.~T. Whiteside, Newton's discovery of the general binomial theorem, The
  Mathematical Gazette (1961) 175--180.

\bibitem{whiteside1961henry}
D.~T. Whiteside, Henry briggs: The binomial theorem anticipated, The
  Mathematical Gazette 45~(351) (1961) 9--12.

\bibitem{polking1972leibniz}
J.~C. Polking, A leibniz formula for some differentiation operators of
  fractional order, Indiana University Mathematics Journal 21~(11) (1972)
  1019--1029.

\bibitem{mazkewitsch1963n}
D.~Mazkewitsch, The n-th derivative of a produc, The American Mathematical
  Monthly 70~(7) (1963) 739--742.

\bibitem{dybowski2009generalization}
R.~Dybowski, A generalization of the leibniz rule for derivatives (2009).

\bibitem{bol1983mathematical}
L.~Bol’shev, N.~Smirnov, Mathematical statistics tables, Moscow: Science
  (1983).

\bibitem{miller1954table}
J.~C.~P. Miller, Table of binomial coefficients, Vol.~3, Published for the
  Royal Society at the University Press, 1954.

\bibitem{royal1954mathematical}
R.~S. G. B. M.~T. Committee, J.~Miller, Mathematical Tables: Table of Binomial
  Coefficients, Published for the Royal Society at Cambridge UP, 1954.

\bibitem{olofsson2012probability}
P.~Olofsson, M.~Andersson, Probability, statistics, and stochastic processes
  (2012).

\bibitem{Oresme}
N.~Oresme, Quaestiones super geometriam Euclidis, Vol.~3, Brill Archive, 1961.

\bibitem{Mengoli}
P.~Mengoli, "praefatio [preface]". novae quadraturae arithmeticae, seu de
  additione fractionum [new arithmetic quadrature (i.e., integration), or on
  the addition of fractions]. bologna: Giacomo monti. (1650).

\bibitem{JohannBernoulli}
J.~Bernoulli, "corollary iii of de seriebus varia". opera omnia. lausanne \&
  basel: Marc-michel bousquet \& co. 4 (1742) 8.

\bibitem{JacobBernoulli1}
J.~Bernoulli, Propositiones arithmeticae de seriebus infinitis earumque summa
  finita [arithmetical propositions about infinite series and their finite
  sums]. basel: J. conrad. (1689).

\bibitem{JacobBernoulli2}
J.~Bernoulli, Ars Conjectandi, Opus Posthumum; Accedit Tractatus De Seriebus
  Infinitis, Et Epistola Gallic{\`e} scripta De Ludo Pilae Reticularis [Theory
  of inference, posthumous work. With the Treatise on infinite series…].,
  Thurnisii, 1713.

\end{thebibliography}

%\printbibliography
\end{document}